\documentclass[12pt]{article}
\usepackage[english]{babel}
\usepackage{epsfig}
\usepackage{amssymb,amsmath,amsfonts,soul,amsthm,enumerate}
\usepackage{graphicx}
\usepackage[utf8]{inputenc}
\usepackage{array}
\usepackage{url}
\usepackage{mathrsfs}
\usepackage{xcolor}
\usepackage[section]{placeins}
\newtheorem{definition}{Definition}

\newtheorem{proposition}{Proposition}

\newcommand{\Lo}{\mathbb{L}}
\newcommand{\Ho}{\mathbb{H}}
\newcommand{\llnorm}{\mathcal{L}(\ell^2,\ell^2)}

\newcounter{oq}
\def\@evenfoot{}
\def\@oddfoot{}

\begin{document}
\def\@evenhead{\vbox{\hbox to \textwidth{\thepage\leftmark}\strut\newline\hrule}}

\def\@oddhead{\raisebox{0pt}[\headheight][0pt]{%
\vbox{\hbox to \textwidth{\rightmark\thepage\strut}\hrule}}}

\newpage
\normalsize
\def\bibname{\vspace*{-30mm}{\centerline{\normalsize References}}}
\thispagestyle{empty}
\vskip 5 mm

\centerline{\bf A NOTE ON OPEN QUESTIONS ASKED TO ANALYSIS AND}
\centerline{\bf NUMERICS CONCERNING THE HAUSDORFF MOMENT PROBLEM}

\vskip 0.3cm

\centerline{\bf Daniel Gerth and Bernd Hofmann}

\vskip 0.3cm

\noindent{\bf Abstract}
{\small We address facts and open questions concerning the degree of ill-posedness of
  the composite Hausdorff moment problem aimed at the recovery of a
  function $x \in L^2(0,1)$ from elements of the infinite dimensional sequence space $\ell^2$
  that characterize moments applied to the antiderivative of $x$.
  This degree, unknown by now, results from the decay rate
  of the singular values of the associated compact forward operator $A$, which is the
  composition of the compact simple integration operator mapping in $L^2(0,1)$
  and the non-compact Hausdorff moment operator $B^{(H)}$ mapping from $L^2(0,1)$ to $\ell^2$.
  There is a seeming contradiction between (a) numerical computations, which show (even for large $n$)
  an exponential decay of the singular values for $n$-dimensional matrices obtained by discretizing
  the operator $A$, and \linebreak (b) a strongly limited smoothness of the well-known kernel $k$ of the Hilbert-Schmidt
  operator $A^*A$. Fact (a) suggests severe ill-posedness of the infinite dimensional
  Hausdorff moment problem, whereas fact (b) lets us expect the opposite, because exponential
  ill-posedness occurs in common just for $C^\infty$-kernels $k$. We recall arguments for the possible occurrence of a polynomial decay of the singular values of $A$,
  even if the numerics seems to be against it, and discuss some issues in the numerical approximation of non-compact operators.}

\vskip 0.2cm


\vskip 0.2cm


\vskip 0.3cm

\setcounter{figure}{0}

\renewcommand{\thesection}{\large 1}

\section{\large Introduction}
\label{sec:intro}

In the recent paper \cite{GHHK21}, we have dealt with properties of the forward operator $B^{(H)}: L^2(0,1) \to \ell^2$ of the Hausdorff moment problem defined as
\begin{equation} \label{eq:Haus}
[B^{(H)} z]_j:= \int_0^1 t^{j-1}z(t)dt \qquad (j=1,2,...).
\end{equation}
This inverse problem, which can be written as an operator equation
\begin{equation} \label{eq:out}
B^{(H)}\, z\,=\,y,
\end{equation}
aims at the recovery of a function $z \in L^2(0,1)$ from data of the square-summable infinite sequence  $\left\{\int_0^1 t^{j-1}z(t)dt\right\}_{j=1}^\infty$ of monomial moments to $z$.
It has been outlined ibid that $B^{(H)}$ is a non-compact and injective bounded linear operator with non-closed range, which implies that the operator equation \eqref{eq:out} is ill-posed of type~I in the sense
of Nashed \cite{Nashed87}. This means that there is an infinite dimensional subspace of $\ell^2$ on which the inverse of $B^{(H)}$ is continuous. Despite considerable efforts in \cite{GHHK21}, a full description of this subspace is still missing.

\medskip

{\parindent0em {\bf Open Question~\arabic{oq}\stepcounter{oq}:}} {\sl What is the infinite dimensional subspace of $\ell^2$, on which the inverse of the operator $B^{(H)}$ is continuous?}

 \medskip

 In contrast to linear ill-posed problems of type~II with compact forward operators in Hilbert spaces, where the decay rate to zero of the singular values expresses
the degree of ill-posedness in a concise way, the characterization of the strength of ill-posedness is much more complicated (cf.~\cite{HofFlei99,HofKin10,MNH20}) in non-compact cases like \eqref{eq:out} with $B^{(H)}$ from \eqref{eq:Haus}.
As discussed in \cite{GHHK21}, for the Hausdorff operator $B^{(H)}$ it seems to be important that $B^{(H)}(B^{(H)})^*: \ell^2 \to \ell^2$ coincides with the infinite dimensional Hilbert matrix $\Ho=\left( \frac{1}{i+j-1}\right)_{i=1,j=1}^{\infty}$, where the $n$-dimensional segments $\Ho_n=\left( \frac{1}{i+j-1}\right)_{i=1,j=1}^{n}$ are very ill-conditioned $n \times n$-matrices with condition numbers that grow exponentially with $n$
as an impact of the well-known asymptotics  $\sigma_n(\Ho_n) \sim \,\exp(-3.526 n)$ as $n \to \infty$ (see, e.g.,~\cite[Example~3.3]{Beckermann00}).

If, however, $B^{(H)}$ acts in combination with a compact operator, then the composition operator is compact, and its singular value decay rate serves as an appropriate measure for the corresponding ill-posedness.
Such situation occurs if we try to recover $x \in L^2(0,1)$ from moment data applied to the antiderivative of $x$. In the present note, we combine the compact integration operator $J: L^2(0,1) \to L^2(0,1)$ defined as
\begin{equation}\label{eq:J}
[J x](s):=\int_0^s x(t)dt\qquad(0 \le s \le 1)
\end{equation}
with $B^{(H)}$ from \eqref{eq:Haus} in the operator equation
\begin{equation} \label{eq:general}
A x\,=\,y\,,
\end{equation}
where the forward operator is here the composition
\begin{equation} \label{eq:comp}
A:=B^{(H)} \circ J: L^2(0,1) \to \ell^2.
\end{equation}
Taking into account the well-known properties of $J$ and $B^{(H)}$, the composition $A$ from \eqref{eq:comp} turns out to be a compact and injective bounded linear operator with non-closed range. Consequently, there exists a singular system $\{\sigma_i(A),u_i(A),v_i(A)\}_{i=1}^\infty$
with complete orthonormal systems $\{u_i(A)\}_{i=1}^\infty$ in $L^2(0,1)$ and $\{v_i(A)\}_{i=1}^\infty$ in $\ell^2$, respectively, satisfying the conditions $Au_i(A)= \sigma_i(A) v_i(A) \;(i=1,2,...)$ and the positive and non-increasing sequence of singular values
$\{\sigma_i(A)\}_{i=1}^\infty$ with $\lim_{i \to \infty} \sigma_i(A)=0$.
In this context, we recall the following definition from \cite[Definition~8]{HofKin10}, which has been slightly updated. We also refer to \cite{HT97} for the related concept of ill-posedness intervals and further discussions.

\begin{definition} \label{def:degree}
Let $A$ be a compact operator mapping between Hilbert spaces with non-closed range. If there exists a constant $C>0$ and an exponent $0<\kappa<\infty$ such that
\begin{equation} \label{eq:kappa}
C\,i^{-\kappa} \le \sigma_i(A) \quad \forall i \in \mathbb{N}\,,
\end{equation}
we call the operator equation \eqref{eq:general} {\sl moderately ill-posed} of degree at most $\kappa$. If for all $\varepsilon>0$\,, \eqref{eq:kappa} does not hold with $\kappa$ replaced
by $\kappa-\varepsilon$ we call \eqref{eq:general} ill-posed of degree $\kappa$. If \eqref{eq:kappa} does not hold for arbitrarily large $\kappa$, we call \eqref{eq:general} {\sl severely ill-posed}.
Typical behaviour of severely ill-posed equations is {\sl exponential ill-posedness}, which means that there exist positive constants $C_1,C_2$ and an exponent $0<\theta<\infty$ such that the inequality
\begin{equation} \label{eq:exp}
\sigma_i(A) \le C_1\, \exp(C_2\,i^\theta) \quad \forall i \in \mathbb{N}
\end{equation}
holds.
\end{definition}

\medskip

{\parindent0em {\bf Open Question~\arabic{oq}\stepcounter{oq}:}} {\sl Is the operator equation \eqref{eq:general} with the composite operator $A$ from \eqref{eq:comp} moderately ill-posed or, in contrast, even exponentially ill-posed?}

 \medskip

To the best of our knowledge, this question that we ask to analysis and numerics in this note has not yet been answered satisfactorily, despite considerable efforts in the articles \cite{Gerth22,GHHK21,HM21}.
In Section~2 we formulate our current knowledge about analytical results and corresponding suggestions with respect to Open Question~2.
We will formulate in this context an Open Question~3, which meets the possible occurrence of exponential ill-posedness for integral equations kernel functions of limited smoothness.  Numerical computations of the singular value decay for $n \times n$ matrices approximating the operator $A$ from \eqref{eq:comp} will be discussed in Section~3. This decay is of exponential type, even if $n$ is fairly large. However, we reiterate an explanation given in \cite{Gerth22} that such phenomenon does not rule out the possibility of moderate ill-posedness for the original
equation \eqref{eq:general} with $A$ mapping between the infinite dimensional Hilbert spaces $L^2(0,1)$ and $\ell^2$. In Section~4 we formulate some open problems on the relation between the spectrum of non-compact operators and their discrete approximations.

\renewcommand{\thesection}{\large 2}
\section{\large Analysis facts for the infinite dimensional composite\\ Hausdorff moment problem} \label{sec:analysis}

In order to get insight into the structure of the compact operator $A$ from \eqref{eq:comp} with $B^{(H)}$ from \eqref{eq:Haus} and $J$ from \eqref{eq:J}, we recall some corresponding results from \cite{GHHK21}. In this context, the
system $\{L_j\}_{j=1}^\infty$ of normalized shifted Legendre polynomials with the explicit expressions
$$L_j(t)=\frac{\sqrt{2j-1}}{(j-1)!}\left( \frac{d}{dt}\right)^{j-1} t^{j-1}(1-t)^{j-1}\quad(0\leq t\leq 1),$$
which is a complete
orthonormal system in $L^2(0,1)$, plays a prominent role. Taking into account \cite[Proposition~5 and Remark~2]{GHHK21} we can factorize the Hausdorff (monomial) moment operator $B^{(H)}$ as $B^{(H)}=\Lo  \circ Q$, where $\Lo: \ell^2 \to \ell^2$ is an infinite lower triangular matrix occurring in the Cholesky
decomposition $\Ho=\Lo \Lo^*$ of the infinite Hilbert matrix $\Ho$. For an explicit expression of the entries of $\Lo$ we refer to \cite[Proposition~5]{GHHK21}. On the other hand,
the operator $Q: L^2(0,1) \to \ell^2$ defined as
\begin{equation}\label{eq:Q}
[Qx]_j:=\langle x, L_j \rangle_{L^2(0,1)}\quad (j=1,2,...)
\end{equation}
characterizes an isometry.
Consequently, we can factorize the operator equation  \eqref{eq:general} into an inner problem
\begin{equation} \label{eq:inner}
Q\,(Jx)=w\,,
\end{equation}
ill-posed of type~II with the compact operator $Q \circ J$,
and an outer problem
\begin{equation} \label{eq:outer}
\Lo\, w \,=\,y\,
\end{equation}
which is ill-posed of type~I.

Solving the operator equation \eqref{eq:inner} is in the papers \cite{Luetal13,Zhao10} considered as the reconstruction of the derivative $x \in L^2(0,1)$ of the function $Jx$ based on Legendre-type moments $\{\int_0^1 L_j(t)[Jx](t)dt  \}_{j=1}^\infty$,
or in other words as a specific approach of numerical differentiation.
Due to the isometry $Q$ and the well-known singular value asymptotics $\sigma_i(J) \asymp 1/i\,$\footnote{For two decreasing sequences $\{s_i\}_{i=1}^\infty$ and $\{t_i\}_{i=1}^\infty$ of positive numbers, the symbol $s_i \asymp t_i$ denotes
that there are constants $0<\underline c \le \overline c < \infty$
    such that the inequalities
    $$     \underline{c} \,s_{i} \le t_{j} \le \overline{c} \,s_{i} \qquad (i=1,2,...)$$
are valid.} of $J$, we have that $\sigma_i(Q \circ J) \asymp 1/i$ and that the operator equation \eqref{eq:inner} is  moderately ill-posed of degree one.

Since $\Lo$ is the infinite lower triangular matrix of the Cholesky decomposition of the infinite Hilbert matrix $\Ho$, we derive from the properties of $\Ho$ (cf., e.g, \cite{Magnus50}) that the spectrum of the non-compact operator $\Lo$ with $\|\Lo\|_{\mathcal{L}(\ell^2,\ell^2)}=\sqrt{\pi}$ contains the whole interval $[0,\sqrt{\pi}]$  and is purely continuous, because $\Ho$ has no embedded eigenvalues.

Although the complete singular system of $J$ is available as
$$\textstyle \left\{\sigma_i(J)=\frac{2}{(2i-1)\pi},\; u_i(J)=\sqrt{2}\cos{\scriptstyle\left(i-\frac{1}{2}\right)}\pi t,\; v_i(J)=\sqrt{2}\sin{\scriptstyle\left(i-\frac{1}{2}\right)}\pi t \right\}_{i=1}^\infty,$$
we cannot conclude from this to the singular system $\{\sigma_i(A), u_i(A),v_i(A)\}_{i=1}^\infty$ of the compact operator $A=\Lo \,\circ Q \,\circ J\,$ or at least to its singular value asymptotics, because the orthonormal eigensystem
$\{u_i(A)\}_{i=1}^\infty$ is completely unknown and does not seem to have anything to do with the eigensystem  $\{u_i(J)\}_{i=1}^\infty$. A visualization of this is given in Figure \ref{fig:sv_AJ}, where we plot a numerical approximation to $u_{10}(A)$ and $u_{10}(J)$ obtained by calculating the singular value decomposition of $1000\times 1000$ matrix approximations to $A$ and $J$.
\begin{figure}
\includegraphics[width=\linewidth]{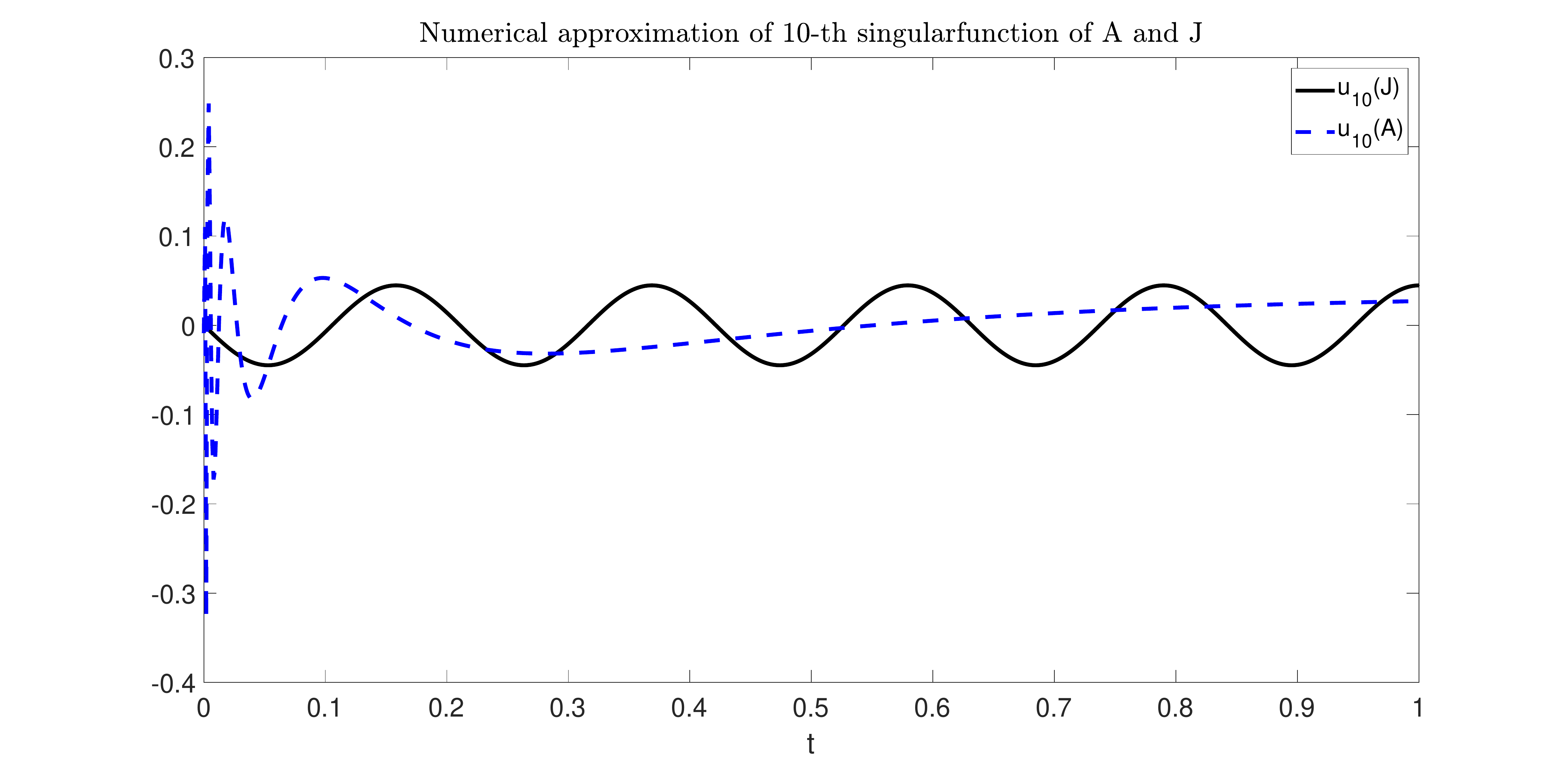}\caption{Singular functions $u_{10}(A)$ and $u_{10}(J)$ computed from the discretized operators. Obviously, the functions look completely different.}\label{fig:sv_AJ}
\end{figure}

We now recall a result from \cite{HM21}, which leaves open whether the operator equation \eqref{eq:general} with the composite operator $A$ from \eqref{eq:comp} is moderately or severely (exponentially) ill-posed in the sense of Definition~\ref{def:degree}.

\begin{proposition}[Corollary~3.6 and Theorem~5.1 in~\cite{HM21}]\label{pro:hm21}
For the composite operator $A: L^2(0,1) \to \ell^2$ from \eqref{eq:comp} there exist positive constants $\underline{C}$ and $\overline{C}$ such that, for sufficiently large indices $i \in \mathbb{N}$, the inequalities
\begin{equation} \label{eq:HM21}
\exp(-\underline{C}\,i) \le \sigma_i(A) \le \frac{\overline{C}}{i^{3/2}}
\end{equation}
concerning the singular values of $A$ are valid.
 \end{proposition}

Hence, it needs further approaches to bridge the big gap of singular value rates occurring in Proposition~\ref{pro:hm21} and to decide whether the lower bound or the upper bound in the inequality chain \eqref{eq:HM21} is more realistic.
Since $A$ is a Hilbert-Schmidt operator, a conceivable alternative approach to the singular value asymptotics of $A$ is via the smoothness of the symmetric and positive kernel $k(s,t)$ for $(s,t) \in [0,1] \times [0,1]$ of the integral operator $A^*A$ mapping in $L^2(0,1)$.
In the appendix of \cite{HM21} it has been proven that this kernel representation attains the form of an infinite series as
\begin{equation} \label{eq:kernel}
[A^*Ax](s)= \int \limits_0^1 k(s,t)\,x(t) dt \quad \mbox{with} \quad k(s,t)=\sum \limits _{j=1}^\infty \frac{(1-s^j)(1-t^j)}{j^2}\,.
\end{equation}

\begin{proposition} \label{pro:kprop}
The kernel $k(s,t)$ from \eqref{eq:kernel} is continuous on the whole closed unit square, i.e.,~$k \in C([0,1]\times [0,1])$. However, there occur poles at some boundary points of the unit square for partial derivatives of $k$ with respect to the variable $s$. Even for the first partial derivative we have that $\frac{\partial k(s,t)}{d s} \notin C([0,1]\times [0,1])$ due to $\lim \limits_{s \to 1} \frac{\partial k(s,t)}{d s} = -\infty$ for $0 \le t<1$.
\end{proposition}
\begin{proof} It is evident that the functions $0\le (1-s^j)(1-t^j)/j^2 \le 1/j^2$ are continuous for all $(s,t) \in [0,1]\times [0,1]$ and for all $j \in \mathbb{N}$, which implies that the series $\sum_{j=1}^\infty \frac{(1-s^j)(1-t^j)}{j^2}$ is uniformly absolutely convergent and that the kernel function $k$ is continuous on the whole closed unit square. By repeated formal partial differentiation of all terms inside the sum with respect to the variable $s$ we see that poles with growing order occur at some boundary points, which contradicts an assumption of infinite differentiability of the kernel $k$. Just the first partial derivative of the form
$\frac{\partial k(s,t)}{d s}= \sum _{j=1}^\infty \frac{-s^{j-1}(1-t^j)}{j}$ does not attain a finite value at the boundary points $(s,t)$ of the unit square characterized by $s=1$ and $0 \le t <1$, because of $\lim \limits_{s \to 1} \frac{\partial k(s,t)}{d s} = -\infty$.\end{proof}

\pagebreak

{\parindent0em {\bf Open Question~\arabic{oq}\stepcounter{oq}:}} {\sl Under which conditions can an operator equation \eqref{eq:general} with a Hilbert-Schmidt operator $A$ mapping from $L^2(0,1)$ into an arbitrary Hilbert space $Y$ with non-closed range $\mathcal{R}(A)$ be
severely (exponentially) ill-posed in the sense of Definition~\ref{def:degree}, provided that the kernel $k \in L^2([0,1]\times[0,1])$ from $A^*A: L^2(0,1) \to L^2(0,1)$ has limited smoothness, which means that $k$ is not infinitely many continuously differentiable on the whole closed unit square?}

\bigskip

Answers to that question, which seems to be an open one at the moment, could help to evaluate the impact of the limited smoothness of the kernel $k$ in \eqref{eq:kernel} on the singular value asymptotics of $A$ from \eqref{eq:comp}. If severely ill-posed equations would require
infinite differentiability of the kernel on the whole square $[0,1]\times [0,1]$, then a polynomial decay rate of $\sigma_i(A)$ could be concluded.

\renewcommand{\thesection}{\large 3}
\section{\large Numerical results for the discretized problem and an attempt to explain} \label{sec:numerics}
If we would know
the singular functions $u_i(A) \in L^2(0,1)$ and $v_i(A) \in \ell^2$, then we could verify the ideal $n \times n$-diagonal matrices
$$\mathbb{D}^A_n=\left(\langle A u_i(A),v_j(A)\rangle_{\ell^2} \right)_{i,j=1}^n={\rm diag}(\sigma_1(A),...,\sigma_n(A))$$
with the largest $n$ singular values $\sigma_i(A)=\langle A u_i(A),v_i(A)\rangle_{\ell^2}$ as diagonal entries.

\bigskip

{\parindent0em {\bf Open Question~\arabic{oq}\stepcounter{oq}
:}} {\sl What are the singular functions $u_i(A) \in L^2(0,1)$ and $v_i(A) \in \ell^2$ of the composite operator $A=B^{(H)}\circ J$ from \eqref{eq:comp}?}

\bigskip

In place of $\mathbb{D}^A_n$, we can  only calculate in practice for orthonormal test bases $\{\hat u_i\}_{i=1}^\infty$ in $L^2(0,1)$ and $\{\hat v_i\}_{i=1}^\infty$ in $\ell^2$ the $n \times n$-Gramian matrices
\begin{equation} \label{eq:Gn}
\mathbb{G}^A_n=\left(\langle A \hat u_i,\hat v_j\rangle_{\ell^2} \right)_{i,j=1}^n \qquad \mbox{with} \qquad  \sigma_i(\mathbb{G}^A_n) \le \sigma_i(A) \quad (i=1,2,...,n)
\end{equation}
(cf.~\cite[Theorem~3.1]{Allen85}) and $\;\lim\limits_{n \to \infty}\sigma_1(\mathbb{G}^A_n) = \sigma_1(A)\;$ (cf.~\cite[Theorem~3.5]{Allen85}).
As a consequence of \eqref{eq:Gn}, for fixed (possibly even large) $n$, the computed matrix singular values $\sigma_i(\mathbb{G}^A_n)\,(i=1,2,...,n)$ by using any test bases $\{\hat u_i\}_{i=1}^\infty$ and $\{\hat v_j\}_{j=1}^\infty$ can only yield lower bounds in (13) of the decay rates of the singular values $\sigma_i(A)$ of the operator $A$ mapping between infinite dimensional spaces. Due to the prominent role of the Legendre polynomials, the Gramian matrices
$\mathbb{G}^A_n=\left(\langle A L_i,e^{(j)}\rangle_{\ell^2)} \right)_{i,j=1}^n $
for $\hat u_i(t)=L_i(t)\;(0 \le t \le 1)$   and $\hat v_j=e^{(j)}\;(e_k^{(j)}=\delta_{k,j})$ are of specific interest. From \cite[Lemma~5.4]{HM21} we have that $\langle A L_1,e^{(j)}\rangle_{\ell^2}= \frac{1}{j+1}$ and for $2 \le i \le n$  that
$\langle A L_i,e^{(j)}\rangle_{\ell^2}= \frac{1}{j \sqrt{2j+1}}\langle L_i,h_j\rangle_{L^2(0,1)}$ by using the normalized monomials $h_j(t)=\sqrt{2j+1}\,t^j$ with $\|h_j\|_{L^2(0,1)}=1$. A plot of the first 20 singular values $\sigma_i(\mathbb{G}_n^A)\;(i=1,\dots,20)$ and $n=100$ is shown in Figure \ref{fig:gram}. There the exponential decay of the singular values is easily seen.

\begin{figure}
\includegraphics[width=\linewidth]{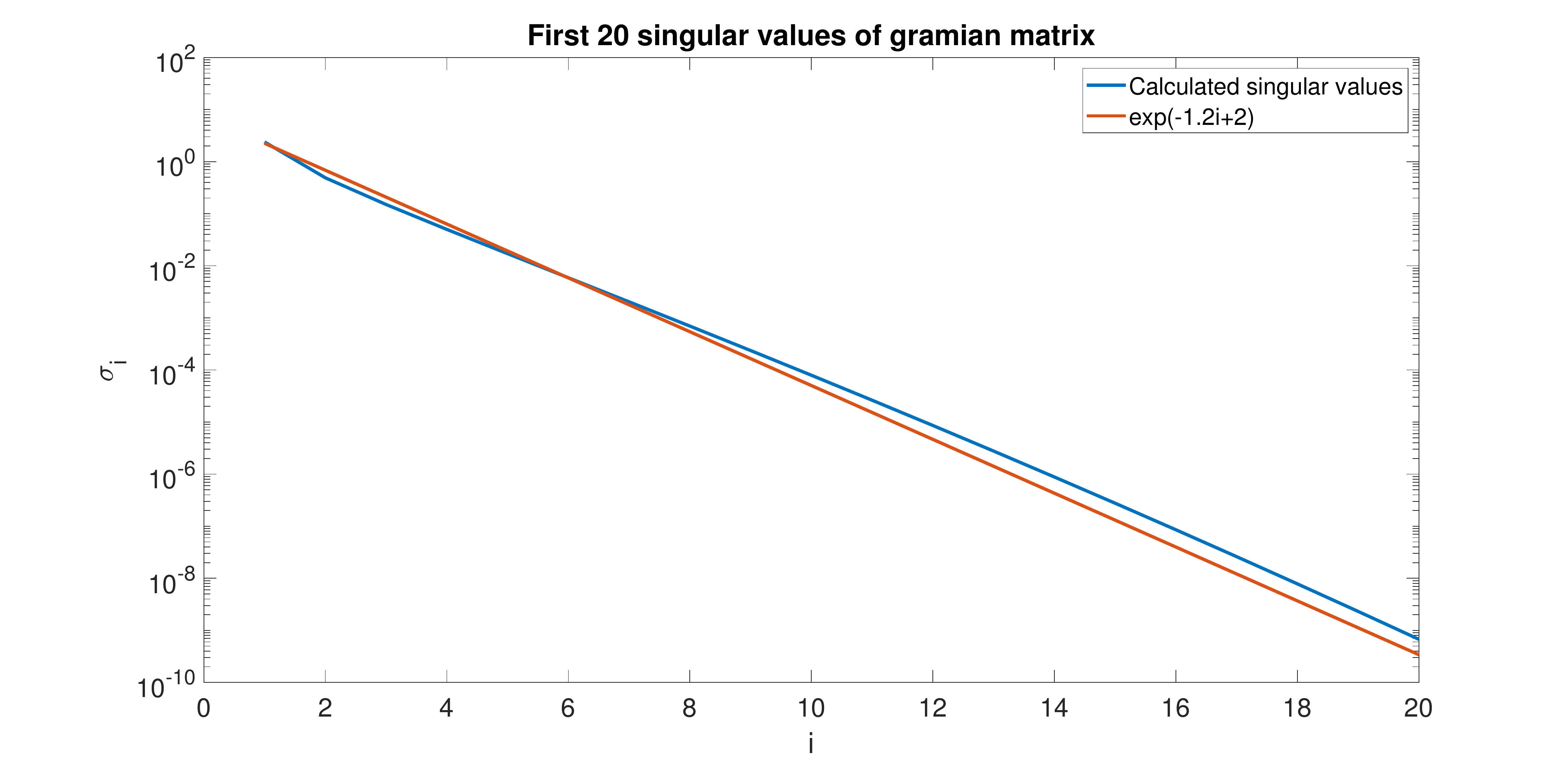}\caption{Semi-logarithmic plot of the first 20 singular values of $\mathbb{G}_n^A$ with $n=100$. As reference we plotted the function $\exp(-1.2i+2)$. The fit of this function to the computed singular values emphasizes the exponential decay of the latter.}\label{fig:gram}
\end{figure}

\pagebreak

The fact that the kernel $k$ from \eqref{eq:kernel} is not continuously differentiable makes {\sl severe ill-posedness} of the operator equation \eqref{eq:general} with the composite operator $A$ from \eqref{eq:comp} {\sl quite unlikely}. In contrast to that, numerical computations of singular values of $n \times n$-matrices arising from {\sl discretizations of} $A$ with $n$ supporting points show clearly an exponential decay, even for large $n$ (displayed with $n=10^4$ in Figure~\ref{fig:1}), and hence {\sl suggest severe ill-posedness}. However, there are no stringent analytical arguments that allow us to conclude to an exponential decay of the singular values of the operator $A: L^2(0,1) \to \ell^2$
from an exponential decay of singular values of associated discretization matrices. Taking into account the extreme ill-conditioning of Hilbert matrices, below we recall arguments from \cite{Gerth22} that might explain the exponential decay of matrix singular values even if the original operator $A$ mapping between infinite dimensional spaces has polynomially decreasing singular values.

\begin{figure}[h!]
\begin{center}
\includegraphics[width=\linewidth]{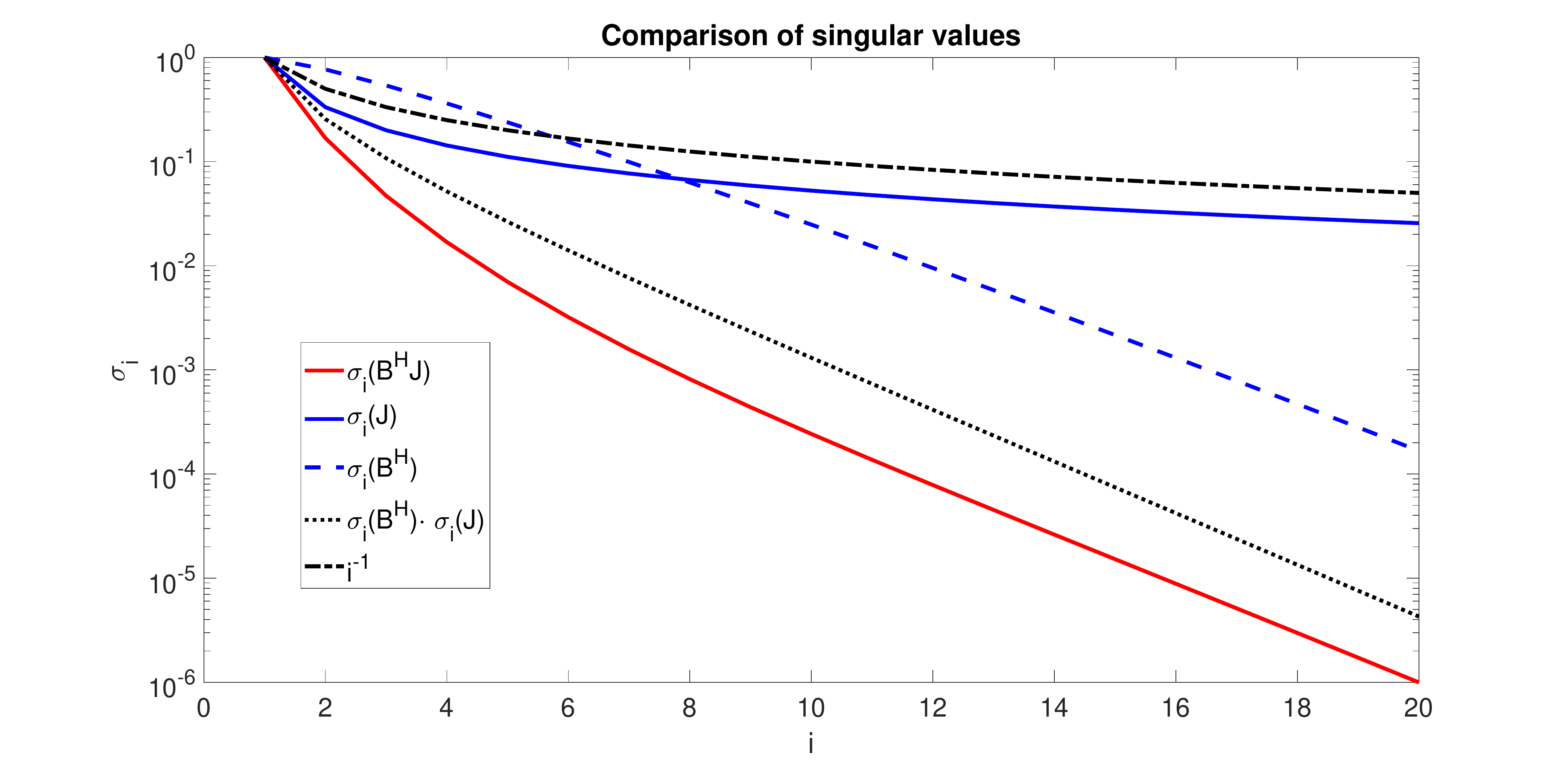}
\caption{(Plot courtesy of \cite{Gerth22}) Semi-logarithmic plot of singular values of $n \times n$-matrices with $n=10^4$ supporting points representing discretization matrices of the operators $A$, $B^{(H)}$ and $J$. While numerically the singular values of $J$ decay as suggested by the theory, the singular values of $B^{(H)}\circ J$ decay exponentially in the numerical experiments.} \label{fig:1}
\end{center}
\end{figure}

In \cite{Gerth22} some arguments were made that an exponential decay of matrix singular values is possible even if the singular values of the original operator $A$ are slowly decreasing, which we recall here. Consider the $n$-dimensional segments $\Ho_n$
of the Hilbert matrix $\Ho$ introduced above. Then the corresponding segments $\Lo_n$ (first $n$ columns and rows of $\Lo$) satisfy the condition $\Lo_n (\Lo_n)^*=\Ho_n$, which implies that $\sigma_i(\Lo_n )=\sqrt{\sigma_i(\Ho_n)}$.
In this context, we recall that (cf., e.g.,~\cite[Theorem~3.5]{Allen85}) $$1 \le \sigma_1(\Ho_n)  \le \pi= \lim\limits_{n \to \infty} \sigma_1(\Ho_n)=\sigma_1(\Ho),$$
and we rewrite the inequality from~\cite[formula (4.8)]{Beckermann00}) as
\begin{equation} \label{eq:Beck1}
\sigma_i(\Lo_n) \le 2\,[\varphi(n)]^{i-1}\,(\sigma_1(\Ho_n))^{1/2} \quad (i=1,2,...,n-1),
\end{equation}
with the damping factor $\varphi(n)$,
\begin{equation} \label{eq:Beck2}
 0< \varphi(n):= \exp \left(-\frac{\pi^2}{2 \ln(8n-4)}\right)<1\,,
\end{equation}
which is growing with $n$ to one with an asymptotics characterized by \linebreak $1-\varphi(n) \sim 1/\ln(n)$  as $n \to \infty$.
We conjecture that \eqref{eq:Beck1} holds approximately as an equation if $n \gg i$, which means that
for fixed and sufficiently large $n$ the decay of $\sigma_i(\Lo_n)$ with respect to growing $i$ is exponentially as
$\sigma_i(\Lo_n) \sim  \exp(-K\,i)\;$ with some $K=K(n)>0$. Table~\ref{tab:1} shows that the multiplier $\varphi(N)^{i-1}$
in \eqref{eq:Beck1} is rather far from one even for $i$ of medium size and large $n$.

\bigskip

\begin{table}[h!]
  \begin{center}
    \begin{tabular}{ >{\raggedleft\arraybackslash}p{1.5cm} p{1.5cm} p{1.5cm} p{1.5cm} p{2.5cm}}
      $\;\;\;\;\;\;n$ & $i=2$ & $i=4$ &   $i=10$ & $i=51$  \\
      \\
      \hline
      \\
     \;\;\;\;$10^2$  & 0.4777 & 0.1091 & 0.0013 & $9.1932\cdot 10^{-17} $  \\
      \;\;$10^3$  & 0.5774 & 0.1926 & 0.0071 &$1.1920\cdot 10^{-12}$  \\
      $10^4$ & 0.6459       &    0.2695    &   0.0196 & $3.2240\cdot 10^{-10}$        \\
      $10^6$ & 0.7331 &  0.3940 & 0.0612 & $1.8129\cdot 10^{-7}$\\
      $10^9$ & 0.8054 & 0.5224  & 0.1426 & $1.9982\cdot 10^{-5}$
    \end{tabular}
    \vspace{5mm}
    \caption{Values of occurring multiplier $\varphi(n)^{i-1}$ in \eqref{eq:Beck1}.} \label{tab:1}
      \end{center}
\end{table}

It was conjectured in \cite{Gerth22} that $\lim_{n\rightarrow \infty} \sigma_i(\Ho_n)=\pi$ for any fixed $i$ as it appears that even numerically the singular values of the truncated Hilbert matrix (and hence those of the operators $\Lo_n$) increase slowly as the truncation index grows. For the multiplication operator and its discretization an analogon to this conjecture was shown in \cite{Gerth22}, but here we must leave it as an open question, since we only have the upper bound \eqref{eq:Beck1} on the singular values, but no lower bound.

\medskip

{\parindent0em {\bf Open Question~\arabic{oq}\stepcounter{oq}:}} {\sl Is $\lim_{n\rightarrow \infty} \sigma_i(\Ho_n)=\pi$ for any fixed $i$, or, if not, what is the behaviour of $\sigma_i(\Ho_n)$ as $n$ increases?}
 \smallskip

To apply these findings for the interpretation of the singular value decay curves in Figure~\ref{fig:1}, we denote by $A_n$ and $J_n$ $n \times n$-discretization matrices of the operators $A=\Lo \circ Q \circ J$ and $J$, respectively.
Then we know that $\sigma_{2i}(A_n) \le \sigma_i(\Lo_n)\, \sigma_i(J_n),$ but the associated curves of Figure~\ref{fig:1} even suggest that we have for large $n$ approximately
\begin{equation} \label{eq:inter}
\sigma_{i}(A_n) \approx  \sigma_i(\Lo_n)\, \sigma_i(J_n) \approx C\,\exp(-K(n)\,i)\,\sigma_i(J_n),
\end{equation}
where $\sigma_i(J_n) \sim 1/i$ holds.
If $n$ is not too large, then for medium values of $i$ the singular value $\sigma_i(A_n)$ is dominated by $\sigma_i(\Lo_n) \sim \exp(-K(n)\,i)$, because the damping multiplier $\varphi(n)^{i-1}$ in \eqref{eq:Beck1} is still far from one. This explains the exponential decay in the $\sigma_{i}(A_n)$-curve of Figure~\ref{fig:1} independent of the objective singular value decay rate of the operator $A$ mapping between infinite dimensional spaces. For making assertions
concerning this decay rate, numerics reach its limit here. Only for small $i$ and very large $n$ we have that $\varphi(n)^{i-1}$ is close to one, which would reflect a polynomial decay of $\sigma_i(J_n)$ in $\sigma_i(A_n)$
in a realistic way.  One might argue that this is visible in Figure \ref{fig:1} for $i<5$.

\renewcommand{\thesection}{\large 4}
\section{\large The relationships between the spectrum of non-compact operators and their discrete approximations} \label{sec:noncomp}
The argument presented above boils down to the convergence of the matrix approximations $\Lo_n$ to $B^H$, and the relation of the spectra. As before, we recall the relations $\Lo_n\Lo_n^*=\Ho_n$ and $\Lo\Lo^*=\Ho$ and discuss the Hilbert matrix in the following. To compare $\Ho$ and $\Ho_n$ we extend here $\Ho_n$ to an infinite matrix by filling it up with zero rows and columns such that both are operators mapping in $\ell^2$.
In this sense, $\Ho_n$ has a finite dimensional range and hence represents a compact operator, whereas $\Ho$ is not compact. Therefore, $\Ho_n$ cannot converge to $\Ho$ in norm, i.e.~$\|\Ho-\Ho_n\|_{\llnorm}\not\to 0$ as $n \to \infty$,
because the norm limits of compact operators are always compact. Related to this are the spectra, where $\Ho_n$ possesses a discrete spectrum for all $n \in \mathbb{N}$, but the spectrum of $\Ho$ is purely continuous. This raises the following question that seems to be closely related to Open Question 4.
\medskip

{\parindent0em {\bf Open Question~\arabic{oq}\stepcounter{oq}:}} {\sl How does $\|\Ho-\Ho_n\|_{\llnorm}$ behave as function of $n$?}
 \smallskip

The truncated matrices $\Ho_n$ have been treated in the literature fairly extensively, but we could not find publications on the remainder. Hence we can only speculate here, and it seems plausible that $\|\Ho-\Ho_n\|_{\llnorm}=\pi=\|\Ho\|_{\llnorm}$. To support this idea, we show in Figure \ref{fig:operator_normconvergence} a plot of the norms $\|\Ho_n-\Ho_{100}\|_{\mathcal{L}(\ell^2(n),\ell^2(n))}$ that we computed in the discrete finite dimensional setting with MATLAB. As $n$ increases the norm increases. For the largest $n=12000$ we find $\|\Ho_n-\Ho_{100}\|_{\mathcal{L}(\ell^2(n),\ell^2(n))}\approx 2.19$, which is still significantly smaller than $\pi$, but in terms of the slow convergence of the term $\varphi(n)$ from \eqref{eq:Beck2}, we might need an unreasonable large $n$ to get close to $\pi$.

\begin{figure}
\includegraphics[width=\linewidth]{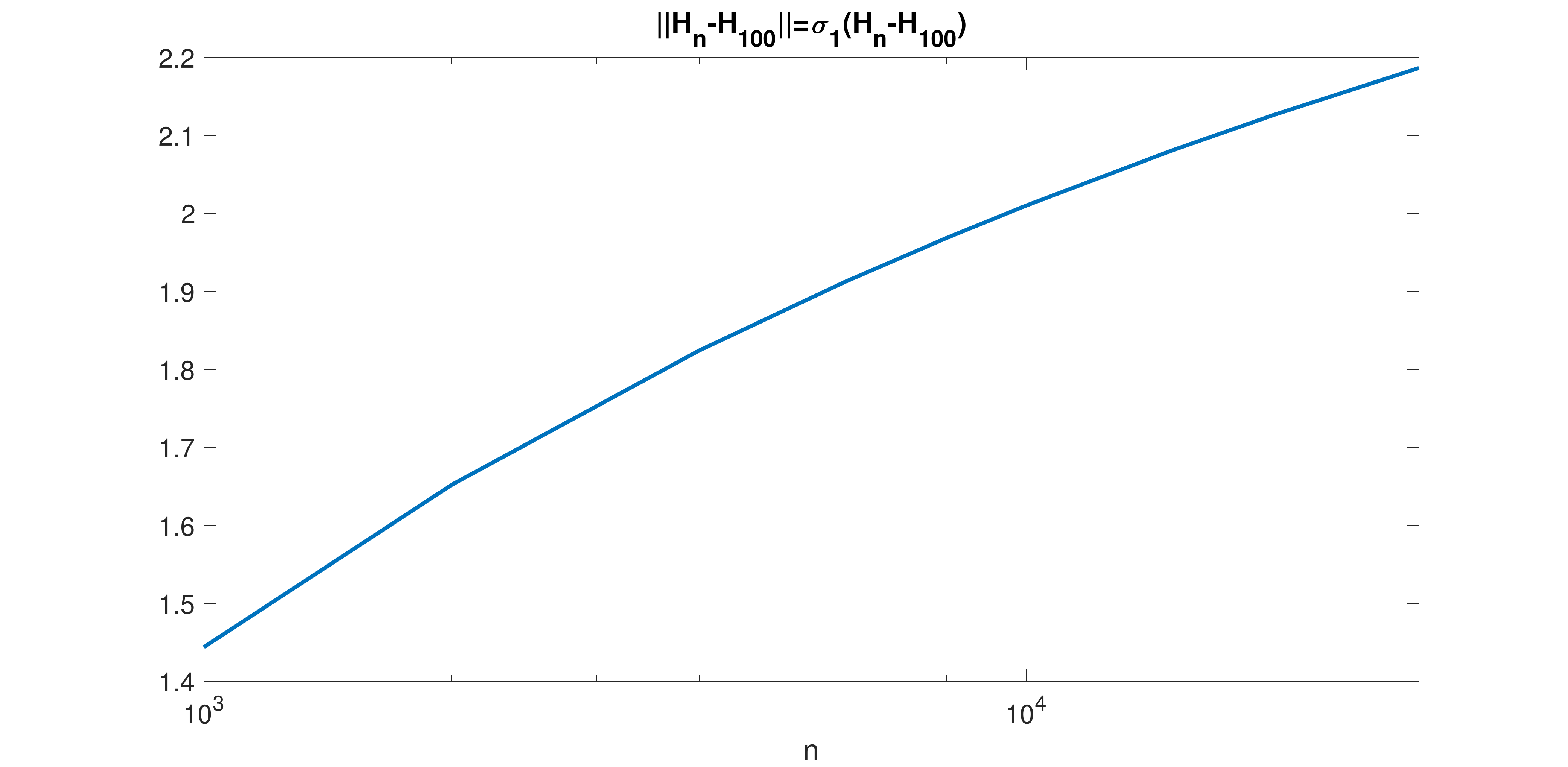}\caption{Numerical computation of $\|\Ho_n-\Ho_{100}\|_{\mathcal{L}(\ell^2(n),\ell^2(n))}$ for $1000\leq n\leq 12000$.}\label{fig:operator_normconvergence}
\end{figure}

\section*{Acknowledgment}
Daniel Gerth has been supported in part by the German Science Foundation (DFG) under the grant GE 3171/1-1 (Project No. 416552794) and by the European Social Fund in conjunction with the German Federal State of Saxony as part of the REACT Junior Research Group project ELIOT (100602771). Bernd Hofmann has been supported by the German Science Foundation (DFG) under the grant HO~1454/13-1 (Project No.~453804957).


\begin{flushleft}

Daniel Gerth,\\
Chemnitz University of Technology, \\
Faculty of Mathematics, 09107 Chemnitz, Germany,\\
Email: {\tt daniel.gerth@mathematik.tu-chemnitz.de},\\

\smallskip

Bernd Hofmann,\\
Chemnitz University of Technology, \\
Faculty of Mathematics, 09107 Chemnitz, Germany,\\
Email: {\tt bernd.hofmann@mathematik.tu-chemnitz.de},\\

\end{flushleft}

\end{document}